\documentclass[12pt]{article}
\usepackage{amsmath,amssymb,amsfonts,amsthm}
\newenvironment{myabstract}{\par\noindent
{\bf Abstract . } \small }
{\par\vskip8pt minus3pt\rm}
\newcounter{item}[section]
\newcounter{kirshr}
\newcounter{kirsha}
\newcounter{kirshb}
\newenvironment{enumarab}{\setcounter{kirshb}{1}
\begin{list}{(\arabic{kirshb})}{\usecounter{kirshb}} }{\end{list}}
\newtheorem{theorem}{Theorem}[section]

\newtheorem{lemma}[theorem]{Lemma}
\newtheorem{corollary}[theorem]{Corollary}
\theoremstyle{definition}

\newtheorem{definition}[theorem]{Definition}

\def\C{{\mathfrak{C}}}

\def\At{{\bf At}}
\def\Nr{{\mathfrak{Nr}}}

\def\A{{\mathfrak{A}}}
\def\B{{\mathfrak{B}}}
\def\C{{\mathfrak{C}}}

\def\N{{\mathfrak{N}}}

\def\CA{{\bf CA}}

\def\K{{\bf K}}
\def\K{{\bf K}}

\def\RCA{{\bf RCA}}

\def\(R)RA{{\bf (R)RA}}

\def\c #1{{\cal #1}}
 \def\CA{{\sf CA}}
\def\B{{\sf B}}

\def\w{{\sf w}}
\def\y{{\sf y}}
\def\g{{\sf g}}

\def\r{{\sf r}}
\def\K{{\sf K}}

 \def\Cm{{\mathfrak{Cm}}}
\def\Nr{{\mathfrak{Nr}}}
\def\SNr{{\bf S}{\mathfrak{Nr}}}

\def\restr #1{{\restriction_{#1}}}

\def\set#1{\{#1\} }

\def\Nr{{\mathfrak{Nr}}}
\def\Tm{{\mathfrak{Tm}}}
\def\A{{\mathfrak{A}}}
\def\B{{\mathfrak{B}}}
\def\C{{\mathfrak{C}}}

\def\A{{\mathfrak{A}}}
\def\B{{\mathfrak{B}}}
\def\C{{\mathfrak{C}}}

\def\GG{{\mathfrak{GG}}}

\def\At{{\mathfrak{At}}}

\def\CA{{\bf CA}}

\def\RCA{{\bf RCA}}

\def\At{{\sf{At}}}
\def\N{\mathbb{N}}

\def\c #1{{\cal #1}}

\def\pa{$\forall$}
\def\pe{$\exists$}

\def\nodes{{\sf nodes}}

\def\restr #1{{\restriction_{#1}}}

\def\Nr{{\mathfrak{Nr}}}

\def\CA{{\bf CA}}
\def\RCA{{\bf RCA}}
\def\c#1{{\mathcal #1}}

\def\set#1{ \{#1\}}

\def\pe{$\exists$}
\def\pa{$\forall$}
\def\Cm{{\mathfrak Cm}}

\def\N{{\cal N}}

\def\At{{\sf At}}

\def\Cm{{\sf Cm}}

\def\w{{\sf w}}
\def\g{{\sf g}}
\def\y{{\sf y}}
\def\r{{\sf r}}

\def\ws{winning strategy}

 \def\CA{{\sf CA}}

\def\RCA{{\sf RCA}}

\def\y{{\sf y}}
\def\g{{\sf g}}
\def\r{{\sf r}}
\def\w{{\sf w}}

\def\N{{\mathbb{N}}}

\title{ For $n\geq 3$ finite, and $k>3$, the class $S\Nr_n\CA_{n+k}$ is not atom-canonical}
\author{Tarek Sayed Ahmed\\
Department of Mathematics, Faculty of Science,\\ 
Cairo University, Giza, Egypt.
  }
\begin{document}
\maketitle
\begin{myabstract} We show that for finite $n\geq 3$ and $k>3$, the class $S\Nr_n\CA_{n+k}$ 
is not atom-canonical solving an open problem first officially reported in \cite{HHbook}, 
and re-appearing in the late \cite{1}.

\end{myabstract}

\section{Introduction}

Throughout we follow the notation of \cite{1}, which is in conformity with the notation \cite{HMT1}. 
The main result in this paper solves an open problem first announced by Hirsch and Hodkinson in \cite{HHbook} 
and re-appearing 
in the late \cite{1}. In the first reference the question is attributed to N\'emeti and the present author, 
and in the second reference, it appears in our chapter \cite{Sayed}.

We use the techniques in \cite{Hodkinson} and \cite{r}. 
In particular, our algebra witnessing non atom canonicity of the class in question, will be a finite rainbow cylindric algebra, 
blown up and blurred in the sense of \cite{ANT}. Our first theorem is conditional. 
It draws the conclusion declared in the astract given that a certain finite rainbow cylindric algebra exists, using
an algebra similar to an existing rainbow algebra in the literature constructed by Hodkinson.
The rest of paper shows that such an algebra indeed does exist; and it is fairly simple. A special game is used
to show that this finite algebra does not neatly embed into $4$ extra dimensions. 

The method of Andre\'ka's splitting is used, by splitting the red atoms,  to blow up and blur the algebra giving a new 
atom structure that is
weakly representable but its completion is not even in $S\Nr_n\CA_{n+4}$, hence, in particular, it is not strongly representable.
The idea is that the finite algebra embeds into the complex algebra by taking every red atom (graph) to the join of its copies.
These exist in the complex algebra but not in the term algebra. 

Our proof relies heavily on Hodkinson's construction in \cite{Hodkinson}, 
in fact our term algebra constructed, whose completion is not in $S\Nr_n\CA_{n+4}$ is almost identical to his
weakly representable atom structure constructed in op.cit. Our construction 
has affinity to the blow up and blur constructions in \cite{ANT} 
and also to lemmas 17.32, 17.34, 17.36, in \cite{HHbook}, a typical blow up and blur construction for rainbow relation 
algebras, proving the analagous result for relation algebras.

The idea of a blow up and blur construction is subtle, and not so hard. The technique is strong. 
Assume that $\K\subseteq {\sf L}$. 
One starts with a(n atomic) finite algebra $\A$ that is not in $\K$. 
The algebra is blown up an blurred by splitting some of the atoms each into $\omega$ many copies. This gives a new infinite 
atom  structure $\At$; the finite algebra is blurred at this level, it does not embed into the term algebra based on this atom structure 
$\Tm\At$, which is constructed to be in ${\sf L}$. But it re-appears on the global level, namely in the complex
algebra based on $\At$, which is the completion of $\Tm\At$, by mapping each atom to the join of its copies. The latter is complete, so
this is well defined. If $\A\notin \K$ and $\K$ is closed under forming subalgebras, 
then $\C$ will not be in $\K,$ as well. But $\C$ is the completion of $\Tm\At$, 
so we get an algebra in ${\sf L}$, namely, $\Tm\At$ whose completion, namely $\Cm\At$, that is not in $\K$. 
Hence $\K$ is not atom canonical, and is not closed
under completions.  In our subsequent investigations $\K$ will be the class $S\Nr_n\CA_{n+4}$ and
${\sf L}$ will be $\RCA_n$, $n$ finite $>2$.

\section{Rainbow cylindric algebras}

We will use the rainbow construction to cylindric algebras of finite dimension $>2$.
Let $A$, $B$ be two relational structures. 
We define a cylindric algebra $\CA_{A, B}$ whose atoms are finite coloured graphs.
That is its atom structure is based on the colours:

\begin{itemize}

\item greens: $\g_i$ ($1\leq i<n-2)$, $\g_0^i$, $i\in A$.

\item whites : $\w, \w_i: i<n-2$
\item reds:  $\r_{ij}$ $(i,j\in B)$,

\item shades of yellow : $\y_S: S\subseteq_{\omega}B$, $S=B.$

\end{itemize}

And coloured graphs are:
\begin{definition}
\begin{enumarab}

\item $M$ is a complete graph.

\item $M$ contains no triangles (called forbidden triples)
of the following types:

\vspace{-.2in}
\begin{eqnarray}
&&\nonumber\\
(\g, \g^{'}, \g^{*}), (\g_i, \g_{i}, \w),
&&\mbox{any }i\in n-1\;  \\
(\g^j_0, \g^k_0, \w_0)&&\mbox{ any } j, k\in A\\
\label{forb:pim}(\g^i_0, \g^j_0, \r_{kl})&&\\
\label{forb:match}(\r_{ij}, \r_{j'k'}, \r_{i^*k^*})&&\mbox{unless }i=i^*,\; j=j'\mbox{ and }k'=k^*
\end{eqnarray}
and no other triple of atoms is forbidden.

\item If $a_0,\ldots   a_{n-2}\in M$ are distinct, and no edge $(a_i, a_j)$ $i<j<n$
is coloured green, then the sequence $(a_0, \ldots a_{n-2})$
is coloured a unique shade of yellow.
No other $(n-1)$ tuples are coloured shades of yellow.

\item If $D=\set{d_0,\ldots  d_{n-2}, \delta}\subseteq M$ and
$\Gamma\upharpoonright D$ is an $i$ cone with apex $\delta$, inducing the order
$d_0,\ldots  d_{n-2}$ on its base, and the tuple
$(d_0,\ldots d_{n-2})$ is coloured by a unique shade
$y_S$ then $i\in S.$

\end{enumarab}
\end{definition}
So we are dealing with a class of models $\K$, each member of $\K$
is a coloured graph,
and the defining relations above can be coded in $L_{\omega_1,\omega}$, more precisely,
every green, white, red, atom corresponds to a binary relation, and every $n-1$ colour is coded as an $n-1$ relations,
and the coloured graphs are defined
as the models of a set of an $L_{\omega_1,\omega}$ theory, as presented in \cite{HHbook2}.

Now from  these coloured graphs we define an atom structure of a $\CA_n$.
Let $${\sf J}=\{a: a \text { is a surjective map from $n$ onto some } M\in \K\}.$$
We may write $M_a$ for the element of ${\sf J}$ for which
$a:n\to M$ is a surjection.
Let $a, b\in {\sf J}$. Define the following equivalence relation: $a \sim b$ if and only if
\begin{itemize}
\item $a(i)=a(j)\text { and } b(i)=b(j)$

\item $M_a(a(i), a(j))=M_b(b(i), b(j))$ whenever defined

\item $M_a(a(k_0)\dots a(k_{n-2}))=M(b(k_0)\ldots b(k_{n-1}))$ whenever
defined
\end{itemize}
Let $\At$ be the set of equivalences classes. Then define
$$[a]\in E_{ij} \text { iff } a(i)=a(j)$$
$$[a]T_i[b] \text { iff }a\upharpoonright n\sim \{i\}=b\upharpoonright n\sim \{i\}.$$

This defines a  $\CA_n$
atom structure. By $\CA_{A,B}$ we shall mean any algebra based on this atom structure, context will specify.
Some coloured graphs should deserve special attention:
\begin{definition}
Let $i\in A$, and let $\Gamma$ be a coloured graph  consisting of $n$ nodes
$x_0,\ldots  x_{n-2}, z$. We call $\Gamma$ an $i$ - cone if $\Gamma(x_0, z)=\g^0_i$
and for every $1\leq j\leq n-2$ $\Gamma(x_j, z)=\g_j$,
and no other edge of $\Gamma$
is coloured green.
$(x_0,\ldots x_{n-2})$
is called the center of the cone, $z$ the apex of the cone
and $i$ the tint of the cone.
\end{definition}

Games on these atom structures are the usual atomic games played on networks \cite{HHbook}, \cite{HHbook2},
translated  to coloured graphs. (This is a special case of the correspondence between networks on atom structures of a class of models
and the models, this correspondence also holds for Monk's algebras, defined also as an instance of algebras 
based on atom structures constructed from classes 
of models.)

Let $\delta$ be a map. Then $\delta[i\to d]$ is defined as follows. $\delta[i\to d](x)=\delta(x)$
if $x\neq i$ and $\delta[i\to d](i)=d$. We write $\delta_i^j$ for $\delta[i\to \delta_j]$.
We recall that atomic networks are defined as follows.

\begin{definition}
Let $2\leq n<\omega.$ Let $\C$ be an atomic $\CA_{n}$.
An \emph{atomic  network} over $\C$ is a map
$$N: {}^{n}\Delta\to \At\C$$
such that the following hold for each $i,j<n$, $\delta\in {}^{n}\Delta$
and $d\in \Delta$:
\begin{itemize}
\item $N(\delta^i_j)\leq {\sf d}_{ij}$
\item $N(\delta[i\to d])\leq {\sf c}_iN(\delta)$
\end{itemize}
\end{definition}
Note than $N$ can be viewed as a hypergraph with set of nodes $\Delta$ and
each hyperedge in ${}^{\mu}\Delta$ is labelled with an atom from $\C$.
We call such hyperedges atomic hyperedges.
We write $\nodes(N)$ for $\Delta.$ We let $N$ stand for the set of nodes
as well as for the function and the network itself. Context will help.
We assume that $\nodes(N)\subseteq \N$.

For $n\leq \omega$, $G^n$ denotes the usual atomic game with $n$ rounds \cite{HHbook2}.
Then a \ws\ for \pe\ in $n$ rounds using the unlimited number of pebbles
can be coded in a first order sentence called a Lyndon condition.

\begin{definition}
Let $M\in \K$ be arbitrary. Define the corresponding network $N_{M}$
on $\C$,
whose nodes are those of $M$
as follows. For each $a_0,\ldots a_{n-1}\in M$, define
$N_{M}(a_0,\ldots  a_{n-1})=[\alpha]$
where $\alpha: n\to M\upharpoonright \set{a_0,\ldots a_{n-1}}$ is given by
$\alpha(i)=a_i$ for all $i<n$. Then, as easily checked,  $N_{M}$ is an atomic $\C$ network.
Conversely, let $N$ be any non empty atomic $\C$ network.
Define a complete coloured graph $M_N$
whose nodes are the nodes of $N$ as follows:
\begin{itemize}
\item For all distinct $x,y\in M_N$ and edge colours $\eta$, $M_N(x,y)=\eta$
if and only if  for some $\bar z\in ^nN$, $i,j<n$, and atom $[\alpha]$, we have
$N(\bar z)=[\alpha]$, $z_i=x$ $z_j=y$ and the edge $(\alpha(i), \alpha(j))$
is coloured $\eta$ in the graph $\alpha$.

\item For all $x_0,\ldots x_{n-2}\in {}^{n-1}M_N$ and all yellows $\y_S$,
$M_N(x_0,\ldots x_{n-2})= \y_S$ if and only if
for some $\bar z$ in $^nN$, $i_0, \ldots  i_{n-2}<n$
and some atom $[\alpha]$, we have
$N(\bar z)=[\alpha]$, $z_{i_j}=x_j$ for each $j<n-1$ and the $n-1$ tuple
$\langle \alpha(i_0),\ldots \alpha(i_{n-2})\rangle$ is coloured
$\y_S.$ Then $M_N$ is well defined and is in $\K$.
\end{itemize}
\end{definition}
The following is then, though tedious and long,  easy to  check:

\begin{theorem}
For any $M\in \K$, we have  $M_{N_{M}}=M$,
and for any network
$N$ on $\At$, $N_{{M}_N}=N.$
\end{theorem}
This translation makes the following equivalent formulation of games
played on coloured graphsm, that were originally formulated for networks.

\begin{definition}
The new  game builds a nested sequence $M_0\subseteq M_1\subseteq \ldots $
of coloured graphs.
\pa\ picks a graph $M_0\in \K$ with $M_0$
\pe\  makes no response
to this move. In a subsequent round, let the last graph built be $M_i$.
$\forall$ picks
\begin{itemize}
\item a graph $\Phi\in \K$ with $|\Phi|=n$
\item a single node $k\in \Phi$
\item a coloured graph embedding $\theta:\Phi\sim \{k\}\to M_i$
Let $F=\phi\smallsetminus \{k\}$. Then $F$ is called a face.
\pe\ must respond by amalgamating
$M_i$ and $\Phi$ with the embedding $\theta$. In other words she has to define a
graph $M_{i+1}\in \K$ and embeddings $\lambda:M_i\to M_{i+1}$
$\mu:\phi \to M_{i+1}$, such that $\lambda\circ \theta=\mu\upharpoonright F.$
\end{itemize}
\end{definition}

Now let us consider the possibilities. There may be already a point $z\in M_i$ such that
the map $(k\mapsto z)$ is an isomorphism over $F$.
In this case \pe\ does not need to extend the
graph $M_i$, she can simply let $M_{i+1}=M_i$
$\lambda=Id_{M_i}$, and $\mu\upharpoonright F=Id_F$, $\mu(\alpha)=z$.
Otherwise, without loss of generality,
let $F\subseteq M_i$, $k\notin M_i$.
Let ${M_i}^*$ be the colored graph with nodes $\nodes(M_i)\cup\{k\}$,
whose edges are the combined edges of $M_i$ and $\Phi$,
such that for any $n-1$ tuple $\bar x$ of nodes of
${M_i}^*$, the color ${M_i}^*(\bar x)$ is
\begin{itemize}
\item $M_i(\bar x)$ if the nodes of $x$
all lie in $M$ and $M_i(\bar x)$ is defined
\item $\phi(\bar x)$ if the nodes of $\bar x$ all lie in
$\phi$ and $\phi(\bar x)$ is defined
\item undefined, otherwise.
\end{itemize}
\pe\ has to complete the labeling of $M_i^*$ by adding
all missing edges, colouring each edge $(\beta, k)$
for $\beta\in M_i\sim\Phi$ and then choosing a shade of
yellow for every $n-1$ tuple $\bar a$
of distinct elements of ${M_i}^*$
not wholly contained in $M_i$ nor $\Phi$,
if non of the edges in $\bar a$ is coloured green.
She must do this on such a way that the resulting graph belongs to $\K$.
If she survives each round, \pe\ has won the play.

\section{The main result, its proof and its consequences}

Here we prove our main result. We start by a conditional theorem, that is the soul and heart of our 
proof, a blow up and blur construction.

\begin{theorem}\label{hodkinson}  Let $k\geq 1$ be finite. Assume that there exists a finite 
rainbow cylindric algebra, not in $S\Nr_n\CA_{n+k}$. Then the class $\SNr_n\CA_{n+k}$ is not closed under completions. 
In more detail,  there exists an atom structure $\At$ such that $\Tm\At$ is representable but $\C=\Cm\At\notin 
S\Nr_n\CA_{n+k}$. In particular $\C$ does not have an $n+k$ relativized square representation.
\end{theorem}

\begin{proof}

Let $L^+$ be the rainbow signature consisting of the binary
relation symbols $\g_i :i<n-2 , \g_0^i: i<m, \w, \w_i: i <n-2, \r_{jk}^i  (i<\omega, j<k<n)$; 
and the $(n-1)$ ary-relation symbols 
$\y_S: S\subseteq n+2)$, where the $g_0^i: i<m$ and $r_{jk}$, $j<k<n$, are the greens and reds in the finite algebra, respectively.
We denote by $\GG$ the class of models of the rainbow theory formulated in the above rainbow structure. These are coloured graphs.
Consider an additional label, a shade of red that is also a binary relation $\rho$; it is not in the rainbow 
signature though, but can be used to colour edges in coloured graphs.
Our signature differs from Hodkinson's \cite{Hodkinson} in that the greens and the suffices of the reds are finite.
However, like in \cite{Hodkinson}, one can show using the standard rainbow argument that 
there is a countable $n$ homogeneous model  $M\in \GG$ with the following
property:\\
$\bullet$ If $\triangle \subseteq \triangle' \in \GG$, $|\triangle'|
\leq n$, and $\theta : \triangle \rightarrow M$ is an embedding,
then $\theta$ extends to an embedding $\theta' : \triangle'
\rightarrow M$.  
Here $\rho$ plays a key role, \pe\ uses it whenever she is forced a red, and her \ws\ in the $\omega$ rounded usual atomic games,
enables her to build $M$ in a step by step manner.

Now let $W = \{ \bar{a} \in {}^n M : M \models ( \bigwedge_{i < j < n,
l < n} \neg \rho(x_i, x_j))(\bar{a}) \}$.
Here we are discarding assignments who have a $\rho$ labelled edge.
The logics $L^n, L^n_{\infty \omega}$ are taken in the rainbow
signature.

For an $L^n_{\infty \omega}$-formula $\varphi $,  define
$\varphi^W$ to be the set $\{ \bar{a} \in W : M \models_W \varphi
(\bar{a}) \}$.  Then   set $\A$ to be the relativised set algebra with domain
$$\{\varphi^W : \varphi \,\ \textrm {a first-order} \;\ L^n-
\textrm{formula} \}$$  and unit $W$, endowed with the algebraic
operations ${\sf d}_{ij}, {\sf c}_i, $ ect., in the standard way , and of coures formulas are taken in the suitable 
signature. Set $\C$ to be the with domain
$$\{\varphi^W : \varphi \,\ \textrm {an} \;\ L_{\infty, \omega}^n-
\textrm{formula} \}.$$

The $n$-homogeneity built into
$M$, in all three cases by its construction implies that the set of all partial
isomorphisms of $M$ of cardinality at most $n$ forms an
$n$-back-and-forth system. But we can even go
further.  
Let $\chi$ be a permutation of the set $\omega \cup \{ \rho\}$. Let
$ \Gamma, \triangle \in \GG$ have the same size, and let $ \theta :
\Gamma \rightarrow \triangle$ be a bijection. We say that $\theta$
is a $\chi$-\textit{isomorphism} from $\Gamma$ to $\triangle$ if for
each distinct $ x, y \in \Gamma$,
\begin{itemize}
\item If $\Gamma ( x, y) = \r_{jk}^i$,
\begin{equation*}
\triangle( \theta(x),\theta(y)) =
\begin{cases}
\r_{jk}^{\chi(i)}, & \hbox{if $\chi(i) \neq \rho$} \\
\rho,  & \hbox{otherwise.} \end{cases}
\end{equation*}
\end{itemize}

\begin{itemize}
\item If $\Gamma ( x, y) = \rho$, then
\begin{equation*}
\triangle( \theta(x),\theta(y)) \in
\begin{cases}
\r_{jk}^{\chi(\rho)}, & \hbox{if $\chi(\rho) \neq \rho$} \\
\rho,  & \hbox{otherwise.} \end{cases}
\end{equation*}
\end{itemize}

For any permutation $\chi$ of $\omega \cup \{\rho\}$, $\Theta^\chi$
is the set of partial one-to-one maps from $M$ to $M$ of size at
most $n$ that are $\chi$-isomorphisms on their domains. We write
$\Theta$ for $\Theta^{Id_{\omega \cup \{\rho\}}}$.
For any permutation $\chi$ of $\omega \cup \{\rho\}$, $\Theta^\chi$
is an $n$-back-and-forth system on $M$.

Consider now the finite rainbow algebra. We ssumed that it has exactly the same colours as its blown up and blurred version, by forgetting
the superscripts in the reds, ending up with finitely many, each with double indices, satisfying
the normal consistency condition on triples
of reds. Now we move backwards. Split each red $\r_{ij}$ in the colours of the finite algebra, which are finite of course, 
to $\r_{ij}^l$ $l\in \omega$ ($\omega$ copies of $\r_{ij}$, 
together with a  shade of red $\rho$, so we get the above rainbow signature, and we get 
the same $M$, $\A$ and $\C$, and $\C$ is the completion of $\A$.
Every atom in the relativized set algebra $\A$ is uniquely defined by a {\it $MCA$ formula} \cite{Hodkinson}.
For the rainbow signature, a formula $ \alpha$  of  $L^n$ is said to be $MCA$
('maximal conjunction of atomic formulas') if (i) $M \models \exists
x_0\ldots x_{n-1} \alpha $ and (ii) $\alpha$ is of the form
$$\bigwedge_{i \neq j < n} \alpha_{ij}(x_i, x_j) \land \bigwedge\eta_{\mu}(x_0,\ldots x_{n-1}),$$
where for each $i,j,\alpha_{ij}$ is either $x_i=x_i$ or $R(x_i,x_j)$ a binary relation symbol in the rainbow signature, 
and for each $\mu:(n-1)\to n$, $\eta_{\mu}$ is either $y_S(x_{\mu(0),\ldots x_{\mu(n-2)}})$ for some $y_S$ in the signature, 
if for all distinct $i,j<n$, $\alpha_{\mu(i), \mu(j)}$ is not equality nor green, otherwisde it is
$x_0=x_0$.

A formula $\alpha$  being $MCA$ says that the set it defines in ${}^n M$
is nonempty, and that if $M \models \alpha (\bar{a})$ then the graph
$M \upharpoonright rng (\bar{a})$ is determined up to isomorphism
and has no edge whose label is of the form $\rho$. 
Now we have for any permutation $\chi$ of $\omega \cup \{\rho\}$, $\Theta^\chi$
is an $n$-back-and-forth system on $M$.
Hence, any two
tuples (graphs) satisfying $\alpha$ are isomorphic and one is mapped to the
other by the $n$-back-and-forth system $\Theta$ of partial isomorphisms from $M$ to $M$; they are the {\it same} coloured gaph. 

No $L^n_{\infty \omega}$- formula can distinguish any two graphs satisfying an $MCA$ formula.
So $\alpha$
defines an atom of $\A$ --- it is literally indivisible. Since the
$MCA$ - formulas clearly 'cover' $W$, the atoms defined by them are
dense in $\A$. So $\A$ is atomic. There is a one to one correspondence between $MCA$ formulas and 
$n$ coloured graphs whose edges are not labelled by the shade of red 
$\rho$ (up to isomorphism), so that in particular, as we already know, those coloured graphs are the atoms of the algebra.

To show that $S\Nr_n\CA_{n+4}$ is not atom canonical, one embeds the finite rainbow cylindic algebra 
to the complex algebra $\C$. A red graph is a graph that has at least one red edge. 
Every $a:n\to \Gamma$, where $\Gamma$ is red in the small algebra, is mapped to the join of $\phi^W$, 
where $\phi$ is an $MCA$ formula, corresponding 
to $a':n\to \Gamma'$ in $\C$, such that $\Gamma'$ is a red copy of $\Gamma$, meaning that for all $i<j<n$ $(a(i), a(j))\in \r$,
then $(a'(i), a'(j))\in \r^l$, for some $l\in \omega$. These joins exist in the complex algebra, 
because it is complete (viewed otherwise, we we are working in $L_{\infty, \omega}^n$).
But the term algebra, namely $\A$, survives these precarious joins, only finitely many or cofinitely many joins of 
reds exist in $\A$. This is necessary (and indeed in our case sufficient) for $\A$ to be representable. 

Every other coloured graph, involving no reds, is mapped to itself.
This induces an embedding from the finite algebra $\A$ to the complex algebra $\C.$
Therefore the class $\SNr_n\CA_{n+k}$ is 
not closed under completions, for the term algebra will be representable,
but its completion will not be in $S\Nr_n\CA_{n+k}$.

Note that if $a':n\to \Gamma'$ is a red graph in
the big algebra, then there is a unique $a: n\to \Gamma$, $\Gamma$ a red graph in the small algebra such that $a'$ is a copy of $a$.
Let $T_i^{\C}$ be the accessibility relation corresponding to the $i$ the cylindrifier in the complex algebra, and 
$T_i^{s}$, be that corresponding to the $i$ cyindrfier in the small algebra.
Now for any non-red graph $b$ in the complex algebra, and any $i<n$, set  $([a'], [b])\in T_i^{\C}$ iff $([a], [b])\in T_i^{s}$.
The last is well defined. 
If $b$ is a red graph as well, then take its original and
define cylindrfiers the same way. For other non-red graphs, the cylindrifiers are exactly like in the small algebra.
So any red copy is cylindrically equivalent to its original.

\end{proof}

The rest of the paper is devoted to constructing such a finite rainbow algebra, and the game that witnesses
its non neat-embeddability in the required extra dimensions. First we fix some notation. 
Then we devise a usual atomic game that tests neat 
embeddability into extra dimensions. The game is taken from \cite{r} adapted to the cylindric case; the difference 
from usual atomic games is that \pa\  can use only finitely many pebbles, 
and the number of those determine how far the given algebra neatly embeds. 
This game will be used to show that our constructed rainbow algebra is not in 
$S\Nr_n\CA_{n+4}$, witnessed by a \ws\ for \pa\ in a finite rounded game.

We need some notation.
\begin{definition}\label{subs}
Let $n$ be an ordinal. An $s$ word is a finite string of substitutions $({\sf s}_i^j)$,
a $c$ word is a finite string of cylindrifications $({\sf c}_k)$.
An $sc$ word is a finite string of substitutions and cylindrifications
Any $sc$ word $w$ induces a partial map $\hat{w}:n\to n$
by
\begin{itemize}

\item $\hat{\epsilon}=Id$

\item $\widehat{w_j^i}=\hat{w}\circ [i|j]$

\item $\widehat{w{\sf c}_i}= \hat{w}\upharpoonright(n\sim \{i\}$

\end{itemize}
\end{definition}

If $\bar a\in {}^{<n-1}n$, we write ${\sf s}_{\bar a}$, or more frequently
${\sf s}_{a_0\ldots a_{k-1}}$, where $k=|\bar a|$,
for an an arbitrary chosen $sc$ word $w$
such that $\hat{w}=\bar a.$
$w$  exists and does not
depend on $w$ by \cite[definition~5.23 ~lemma 13.29]{HHbook}.
We can, and will assume \cite[Lemma 13.29]{HHbook}
that $w=s{\sf c}_{n-1}{\sf c}_n.$
[In the notation of \cite[definition~5.23,~lemma~13.29]{HHbook},
$\widehat{s_{ijk}}$ for example is the function $n\to n$ taking $0$ to $i,$
$1$ to $j$ and $2$ to $k$, and fixing all $l\in n\setminus\set{i, j,k}$.]
The folowing is the $\CA$ analogue of \cite[lemma~19]{r}.
\begin{lemma}\label{lem:atoms2}
Let $n<m$ and let $\A$ be an atomic $\CA_n$,
$\A\subseteq_c\Nr_n\C$
for some $\C\in\CA_m$.  For all $x\in\C\setminus\set0$ and all $i_0, \ldots i_{n-1} < m$ there is $a\in\At(\A)$ such that
${\sf s}_{i_0\ldots i_{n-1}}a\;.\; x\neq 0$.
\end{lemma}
\begin{proof}
We can assume, see definition  \ref{subs},
that ${\sf s}_{i_0,\ldots i_{n-1}}$ consists only of substitutions, since ${\sf c}_{m}\ldots {\sf c}_{m-1}\ldots
{\sf c}_nx=x$
for every $x\in \A$.We have ${\sf s}^i_j$ is a
completely additive operator (any $i, j$), hence ${\sf s}_{i_0,\ldots i_{\mu-1}}$
is too  (see definition~\ref{subs}).
So $\sum\set{{\sf s}_{i_0\ldots i_{n-1}}a:a\in\At(\A)}={\sf s}_{i_0\ldots i_{n-1}}
\sum\At(\A)={\sf s}_{i_0\ldots i_{n-1}}1=1$,
for any $i_0,\ldots i_{n-1}<n$.  Let $x\in\C\setminus\set0$.  It is impossible
that ${\sf s}_{i_0\ldots i_{n-1}}\;.\;x=0$ for all $a\in\At(\A)$ because this would
imply that $1-x$ was an upper bound for $\set{{\sf s}_{i_0\ldots i_{n-1}}a:
a\in\At(\A)}$, contradicting $\sum\set{{\sf s}_{i_0\ldots i_{n-1}}a :a\in\At(\A)}=1$.
\end{proof}

\begin{definition}\label{def:hat}
Let $n<m$. For $\C\in\CA_m$, if $\A\subseteq\Nr_n(\C)$ is an
atomic cylindric algebra and $N$ is an $\A$-network, with $\nodes(N)\subseteq m$,  then we define
$\widehat N\in\C$ by
\[\widehat N =
 \prod_{i_0,\ldots i_{n-1}\in\nodes(N)}{\sf s}_{i_0, \ldots i_{n-1}}N(i_0\ldots i_{n-1})\]
$\widehat N\in\C$ is well defined and depends
implicitly on $\C$.
\end{definition}

For networks $M, N$ and any set $S$, we write $M\equiv^SN$
if $N\restr S=M\restr S$, and we write $M\equiv_SN$ 
if the symmetric difference $\Delta(\nodes(M), \nodes(N))\subseteq S$ and
$M\equiv^{(\nodes(M)\cup\nodes(N))\setminus S}N$. We write $M\equiv_kN$ for
$M\equiv_{\set k}N$. By $\A\subseteq_c \B$, we mean that $\A$ is a complete subalgebra of $\B$, that is if $X\subseteq \A$, and $\sum ^{\A}X=1$, then
$\sum ^{\B}X=1$. 

\begin{lemma}\label{lem:hat}\cite[lemma~26]{r}
Let $n<m$ and let $\A\subseteq_c\Nr_{n}\C$ be an
atomic $\CA_n$
\begin{enumerate}
\item For any $x\in\C\setminus\set0$ and any
finite set $I\subseteq m$ there is a network $N$ such that
$\nodes(N)=I$ and $x\;.\;\widehat N\neq 0$.
\item
For any networks $M, N$ if 
$\widehat M\;.\;\widehat N\neq 0$ then $M\equiv^{\nodes(M)\cap\nodes(N)}N$.
\end{enumerate}
\end{lemma}

\begin{proof}
The proof of the first part is based on repeated use of
lemma ~\ref{lem:atoms2}. We define the edge labelling of $N$ one edge
at a time. Initially no hyperedges are labelled.  Suppose
$E\subseteq\nodes(N)\times\nodes(N)\ldots  \times\nodes(N)$ is the set of labelled hyper
edges of $N$ (initially $E=\emptyset$) and 
$x\;.\;\prod_{\bar c \in E}{\sf s}_{\bar c}N(\bar c)\neq 0$.  Pick $\bar d$ such that $\bar d\not\in E$.  
By lemma~\ref{lem:atoms2} there is $a\in\At(\c A)$ such that
$x\;.\;\prod_{\bar c\in E}{\sf s}_{\bar c}N(\bar c)\;.\;{\sf s}_{\bar d}a\neq 0$.  
Include the edge $\bar d$ in $E$.  Eventually, all edges will be
labelled, so we obtain a completely labelled graph $N$ with $\widehat
N\neq 0$.  
it is easily checked that $N$ is a network.
For the second part, if it is not true that
$M\equiv^{\nodes(M)\cap\nodes(N)}N$ then there are is 
$\bar c \in^{n-1}\nodes(M)\cap\nodes(N)$ such that $M(\bar c )\neq N(\bar c)$.  
Since edges are labelled by atoms we have $M(\bar c)\cdot N(\bar c)=0,$ 
so $0={\sf s}_{\bar c}0={\sf s}_{\bar c}M(\bar c)\;.\; {\sf s}_{\bar c}N(\bar c)\geq \widehat M\;.\;\widehat N$.
\end{proof}

\begin{lemma}\label{lem:khat}  
Let $m>n$. Let $\C\in\CA_m$ and let $\A\subseteq_c\Nr_{n}(\C)$ be atomic.
Let $N$ be a network over $\c A$ and $i,j <n$.
\begin{enumerate}
\item\label{it:-i}
If $i\not\in\nodes(N)$ then ${\sf c}_i\widehat N=\widehat N$.

\item \label{it:-j} $\widehat{N Id_{-j}}\geq \widehat N$.

\item\label{it:ij} If $i\not\in\nodes(N)$ and $j\in\nodes(N)$ then
$\widehat N\neq 0 \rightarrow \widehat{N[i/j]}\neq 0$.
where $N[i/j]=N\circ [i|j]$

\item\label{it:theta} If $\theta$ is any partial, finite map $n\to n$
and if $\nodes(N)$ is a proper subset of $n$,
then $\widehat N\neq 0\rightarrow \widehat{N\theta}\neq 0$.
\end{enumerate}
\end{lemma}
\begin{proof}
The first part is easy.
The second
part is by definition of $\;\widehat{\;}$. For the third part suppose
$\widehat N\neq 0$.  Since $i\not\in\nodes(N)$, by part~\ref{it:-i},
we have ${\sf c}_i\widehat N=\widehat N$.  By cylindric algebra axioms it
follows that $\widehat N\;.\;{\sf d}_{ij}\neq 0$.  By lemma~\ref{lem:hat}
there is a network $M$ where $\nodes(M)=\nodes(N)\cup\set i$ such that
$\widehat M\;.\widehat N\;.\;d_{ij}\neq 0$.  By lemma~\ref{lem:hat} we
have $M\supseteq N$ and $M(i, j)\leq {\sf d}_{ij}$.  It follows that $M=N[i/j]$.
Hence $\widehat{N[i/j]}\neq 0$.
For the final part 
(cf. \cite[lemma~13.29]{HHbook}), since there is 
$k\in n\setminus\nodes(N)$, \/ $\theta$ can be
expressed as a product $\sigma_0\sigma_1\ldots\sigma_t$ of maps such
that, for $s\leq t$, we have either $\sigma_s=Id_{-i}$ for some $i<n$
or $\sigma_s=[i/j]$ for some $i, j<n$ and where
$i\not\in\nodes(N\sigma_0\ldots\sigma_{s-1})$.
Now apply parts~\ref{it:-j} and \ref{it:ij} of the lemma.
\end{proof}

Let $F^m$ be the usual atomic $\omega$ rounded game on networks, 
except that the nodes used are $m$ and \pa\ can re use nodes. For an atom structure $\alpha$, we write 
$F^m(\alpha)$ for the game $F^m$ played on $\alpha$. 

The next theorem is the cylindric algebra analogue of theorem 29 (the implication $(3)\implies (4)$) in \cite{r}.

\begin{theorem}\label{thm:n}
Let $n<m$, and let $\A$ be an atomic $\CA_m$
If $\A\in{\bf S_c}\Nr_{n}\CA_m, $
then \pe\ has a \ws\ in $F^m(\At\A)$. In particular, if $\A$ is countable and completely representable, then \pe\ has a \ws\ in $F^{\omega}(\At\A).$
In the latter case since $F^{\omega}(\At\A)$ is equivalent to the usual atomic rounded game on networks,
the converse is also true.
\end{theorem}
\begin{proof}
We use lemmata \ref{lem:hat}, \ref{lem:khat}. For the first part, if $\A\subseteq\Nr_n\C$ for some $\C\in\CA_m$ then \pe\ always
plays networks $N$ with $\nodes(N)\subseteq n$ such that
$\widehat N\neq 0$. In more detail, in the initial round, let \pa\ play $a\in \At \cal A$.
\pe\ plays a network $N$ with $N(0, \ldots n-1)=a$. Then $\widehat N=a\neq 0$.
At a later stage suppose \pa\ plays the usual cylindrifier move
$(N, \langle f_0, \ldots f_{n-2}\rangle, k, b, l)$
by picking a
previously played network $N$, $f_i\in \nodes(N), \;l<n,  k\notin \{f_i: i<n-2\}$,
and $b\leq {\sf c}_lN(f_0,\ldots  f_{l-1}, x, f_{l+1}, \ldots f_{n-2})$.
Let $\bar a=\langle f_0\ldots f_{l-1}, k\ f_{l+1}, \ldots f_{n-2}\rangle.$
Then by the previous lemma, ${\sf c}_k\widehat N\cdot {\sf s}_{\bar a}b\neq 0$,
hence by lemma \ref{lem:khat}, there is a network  $M$ such that
$\widehat{M}.\widehat{{\sf c}_kN}\cdot {\sf s}_{\bar a}b\neq 0$. Hence
$M(f_0,\dots f_{l-1}, k, f_{l+1}, \ldots f_{n-2})=b,$ 
and $M$ is the required response.

For the second part, we have from the first part, 
that $\A\in S_c\Nr_n\CA_{\omega}$, but countable atomic algebras in the latter class are completely representable \cite{Sayed}; 
the result now follows.
\end{proof}

Recal that we denote the rainbow algebra $R(\Gamma)$ in \cite{HHbook2}, by $\CA_{{\sf G}, \Gamma}$, where ${\sf G}$ is the set of greens that 
we allow to vary. $\CA_{n+2,n+1}$ denotes the finite rainbow cylindric algebra,  
based on the complete irrefexive graphs with underlying sets $n+2$, the greens,  and $n+1$ the reds.
Viewed as a pebble game, in each round $0,1\ldots n+2$ \pe\ places a  new pebble  on  element of $n+2$.
The edges relation in $n+1$ is irreflexive so to avoid losing
\pe\ must respond by placing the other  pebble of the pair on an unused element of $n+1$.
After $n+1$ rounds there will be no such element,
and she loss in the next round. 
Hence \pa\ can win the graph game using $n+4$ pebbles.

\begin{theorem}\label{main2} The rainbow algebra $\A=\CA_{n+2, n+1}$ is not in $S\Nr_n\CA_{n+4}$, hence by theorem \ref{hodkinson},
the latter class is not atom-canonical, because $\A$ can be blown up and blurred to 
give an algebra that is representable, but its completion is not
in $S\Nr_n\CA_{n+4}$.

\end{theorem}
\begin{proof} 

We play the usual atomic game $G^{n+4}$ on $\CA_{n+2, n+1}$, with $n+4$ rounds. 
The usual argument, or rather strategy,  is adopted.  \pa\ forces a win on a red clique using his excess of greens by bombarding \pe\ 
with $\alpha$ cones having the same base, where $\alpha<n+2$, the tints of the cones, 
are the superscripts of the greens.   

In his zeroth move, \pa\ plays a graph $\Gamma$ with
nodes $0, 1,\ldots, n-1$ and such that $\Gamma(i, j) = \w_0 (i < j <
n-1), \Gamma(i, n-1) = \g_i ( i = 1,\ldots, n-2), \Gamma(0, n-1) =
\g^0_0$, and $ \Gamma(0, 1,\ldots, n-2) = \y_{n+2}$. This is a $0$-cone
with base $\{0,\ldots , n-2\}$. In the following moves, \pa\
repeatedly chooses the face $(0, 1,\ldots, n-2)$ and demands a node 
$\alpha$ with $\Phi(i,\alpha) = \g_i (i = 0,\ldots,  1-2)$, $(i=1,\ldots n-2)$ and $\Phi(0, \alpha) = \g^\alpha_0$,
in the graph notation -- i.e., an $\alpha$-cone, $\alpha<n+2$,  on the same base.
\pe\ among other things, has to colour all the edges
connecting new nodes created by \pa\ as appexes of cones based on the face $(0,1,\ldots n-2)$. By the rules of the game
the only permissible colours would be red. Using this, \pa\ can force a
win, for after $n+4$ rounds \pe\ is forced  reds whose indices must match.
Now assume for contradiction that $\A\in S\Nr_n\CA_{n+4}$. 
Then $\A^+=\A\in S_c\Nr_n\CA_{n+4}$, 
hence \pe\ can win the  game $F^{n+4}$ by theorem \ref{thm:n}, so clearly she can also win the game 
$G^{n+4}$ which is impossible.
\end{proof}

From theorems \ref{hodkinson} and \ref{main2}, we readily infer:
\begin{corollary}The following hold  for finite $n\geq 3$:

\begin{enumarab}
\item There exist two atomic
cylindric algebras of dimension $n$  with the same atom structure, only one of which is
representable, the other is not in $S\Nr_n\CA_{n+4}$.

\item  For all $k\geq 4$,  $S\Nr_n\CA_{n+k}$
is not closed under completions and is not atom-canonical. In particular, $\RCA_n$ is not atom-canonical.

\item There exists a non-representable $\CA_n$, that is not even in $S\Nr_n\CA_{n+4}$, with a dense representable
subalgebra.

\item $S\Nr_n\CA_{n+k}$
is not Sahlqvist axiomatizable for every $k\geq 4$. In particular, $\RCA_n$ is not Sahlqvist axiomatizable.

\item There exists an atomic representable $\CA_n$ with no relativized complete $n+k$ square
representation.

\end{enumarab}
\end{corollary}

\begin{proof} We refer to \cite{Hodkinson} where an all rounded picture of the connections of such notions are given.
\end{proof}
Finally, we mention that the above construction applies to many cylindric-like algebras, 
like polyadic algebras with and without equality, and also their proper common reduct called Pinter's substitution 
algebras. This follows from the simple observation that our algebras are generated by elements whose dimernsion sets are $<n$.
Also, if we consider the Pinter's algebra based on $n+2$ greens and $n+1$ reds defined analogously to the cylindric case, then
\pa\ can win the same last game, using the same strategy.

\end{document}